\documentclass[11pt]{article}
\usepackage{amsfonts,amsthm, amsmath, amssymb}
\usepackage{epsfig}
\usepackage{subfigure}
\usepackage{makeidx}
\usepackage{graphicx}

\makeindex
\newenvironment{namelist}[1]{%
\begin{list}{}
{

\settowidth{\labelwidth}{#1}
\setlength{\leftmargin}{1.1\labelwidth}
}
}{%
\end{list}}
\newcommand{\ncom}{\newcommand}
\ncom{\ul}{\underline}
\ncom{\beq}{\begin{equation}}
\ncom{\eeq}{\end{equation}}
\ncom{\bea}{\begin{eqnarray*}}
\ncom{\eea}{\end{eqnarray*}}
\ncom{\beqa}{\begin{eqnarray}}
\ncom{\eeqa}{\end{eqnarray}}
\ncom{\nno}{\nonumber}
\ncom{\non}{\nonumber}
\ncom{\ds}{\displaystyle}
\ncom{\half}{\frac{1}{2}}
\ncom{\mbx}{\makebox{.25cm}}
\ncom{\hs}{\mbox{\hspace{.25cm}}}
\ncom{\rar}{\rightarrow}
\ncom{\Rar}{\Rightarrow}
\ncom{\noin}{\noindent}
\ncom{\bc}{\begin{center}}
\ncom{\ec}{\end{center}}
\ncom{\sz}{\scriptsize}
\ncom{\rf}{\ref}
\ncom{\s}{\sqrt{2}}
\ncom{\sgm}{\sigma}
\ncom{\Sgm}{\Sigma}
\ncom{\psgm}{\sigma^{\prime}}
\ncom{\dt}{\delta}
\ncom{\Dt}{\Delta}
\ncom{\lmd}{\lambda}
\ncom{\Lmd}{\Lambda}
\ncom{\Th}{\Theta}
\ncom{\e}{\eta}
\ncom{\eps}{\epsilon}
\ncom{\pcc}{\stackrel{P}{>}}
\ncom{\lp}{\stackrel{L_{p}}{>}}
\ncom{\dist}{{\rm\,dist}}
\ncom{\sspan}{{\rm\,span}}
\ncom{\re}{{\rm Re\,}}
\ncom{\im}{{\rm Im\,}}
\ncom{\sgn}{{\rm sgn\,}}
\ncom{\ba}{\begin{array}}
\ncom{\ea}{\end{array}}
\ncom{\hone}{\mbox{\hspace{1em}}}
\ncom{\htwo}{\mbox{\hspace{2em}}}
\ncom{\hthree}{\mbox{\hspace{3em}}}
\ncom{\hfour}{\mbox{\hspace{4em}}}
\ncom{\vone}{\vskip 2ex}
\ncom{\vtwo}{\vskip 4ex}
\ncom{\vonee}{\vskip 1.5ex}
\ncom{\vthree}{\vskip 6ex}
\ncom{\vfour}{\vspace*{8ex}}
\ncom{\norm}{\|\;\;\|}
\ncom{\integ}[4]{\int_{#1}^{#2}\,{#3}\,d{#4}}
\ncom{\vspan}[1]{{{\rm\,span}\{ #1 \}}}
\ncom{\dm}[1]{ {\displaystyle{#1} } }
\ncom{\ri}[1]{{#1} \index{#1}}

\def\E{{\mathbb E}}
\def\P{{\mathbb P}}

\newtheorem{theorem}{\bf Theorem}[section]
\newtheorem{remark}{\bf Remark}[section]
\newtheorem{proposition}{Proposition}[section]

\newtheorem{corollary}{Corollary}[section]

\newtheoremstyle
    {remarkstyle}
    {}
    {11pt}
    {}
    {}
    {\bfseries}
    {:}
    {     }
    {\thmname{#1} \thmnumber{#2} }

\theoremstyle{remarkstyle}

\begin{document}

\newpage
\begin{center}
{\Large \bf First-Exit Times of an Inverse Gaussian Process}\\
\end{center}
\vone
\begin{center}
{\bf P. Vellaisamy$^{a}$ and A. Kumar$^{b}$}\\
{\it $^{a}$Department of Mathematics,
Indian Institute of Technology Bombay,\\Powai, Mumbai-400076, India.}\\
{\it $^{b}$Department of Mathematics,
Indian Institute of Technology Madras,\\Adyar, Chennai-600036, India.}\\
\end{center}

\vone

\vtwo
\begin{center}
\noindent{\bf Abstract}
\end{center}
The first-exit time process of an inverse Gaussian L\'evy process is considered. 
The one-dimensional distribution functions of the process are obtained. They are not infinitely divisible and the tail probabilities decay exponentially. These distribution functions can also be viewed as distribution functions of supremum of the Brownian motion with drift. The density function is shown to solve a fractional PDE and the result is also generalized to tempered stable subordinators. The subordination of this process to the Brownian motion is considered and the underlying PDE of the subordinated process is obtained. The infinite divisibility of the first-exit time of a $\beta$-stable subordinator is also discussed.
\vone
\vtwo

\section{Introduction}
The first-exit time process is a stopping time process which arises naturally in diverse fields such as finance, insurance, process control and survival analysis (Lee and Whitmore (2006)). The inverse Gaussian process and its extensions
(see Bandorff-Nielson (1997), Kumar and Vellaisamy (2012)) serve as important probabilistic models for analyzing financial data.
When the total claim size $S(t)$ of an insurance company follows inverse Gaussian process (see e.g. Dufresne and Gerber (1993), Gerber (1992) and the references therein), then the exit-time process corresponds to the situation when the aggregate claim  for that company exceeds certain level. Indeed, this is closely related to the probability of ruin also.
As another example, the first time  an asset price process reaches a certain level is of interest to an investment firm.
The exit-time models are also applied to expected lifetimes of mechanical devices, where the device breaks down when the process reaches an adverse threshold state for the first time.\\
\noindent The hitting time process of the standard Brownian motion has been well studied in the literature and it is  also called a L\'evy subordinator (see e.g. Applebaum, 2009). The hitting time for a general Gaussian process is obtained by Decreusefond and Nualart (2008). Many authors have discussed the first-exit time process for strictly increasing L\'evy processes (see e.g. Meerschaert and Scheffler (2008), Veillette and Taqqu (2010a)). However, an explicit expression for the density of first-exit time process is not possible in many cases (see e.g. Veillette and Taqqu (2010b)). The focus of most of these authors are on the Laplace transform of exit time density or on mean first-hitting time process (see e.g. Veillette and Taqqu (2010b)). In this paper, our aim is to discuss the first-exit time process for an inverse Gaussian (IG) L\'evy process and investigate its various other properties. We provide an alternative way of obtaining the density function of supremum of Brownian motion with drift. Note that inverse Gaussian model and its extensions have been found quite useful
for modeling  financial data (see Barndorff-Nielsen (1997) ). \\

\noindent The density function of an inverse Gaussian distribution IG$(a, b)$, with parameters $a$ and $b$ $(a, b>0)$, is given by
\beq
f(x;a,b) = \displaystyle{(2\pi)}^{-1/2} a x^{-3/2}e^{ab-\frac{1}{2}(a^2 x^{-1} + b^2 x)},~~x>0.
\eeq
It is well known that inverse Gaussian distributions, and generalized inverse Gaussian (GIG) distributions, are generalized gamma convolutions
(see Halgreen, 1979) which are infinitely divisible.
The IG process $\{G(t)\}_{t>0}$  is defined as
(see Applebaum, 2009, p.\ 54)
\beq\label{ig}
G(t) = \inf\{s > 0; B(s) + \gamma s > \delta t \},
\eeq
where $B(t)$ is the standard Brownian motion. That is, $G(t)$ denotes the first time when the Brownian motion with
drift $\gamma$ hits the barrier $\delta t$.
The IG subordinator $G(t)$ is a non-decreasing L\'evy process such
that the increment $G({t+s}) - G(s) $ has an inverse Gaussian IG$(\delta t, \gamma)$ distribution.
The L\'evy measure $\pi$
 corresponding to the inverse Gaussian subordinator is given by
(see e.g. Cont and Tankov (2004))
\beq\label{lmeas}
\pi(dx) = \frac{\delta}{\sqrt{2\pi x^3}}e^{-\frac{\gamma^2x}{2}} I_{\{x>0\}}dx.
\eeq
Note that almost all sample paths of $G(t)$ are strictly increasing with jumps, since the sample paths of $B(t)+\gamma t$ are continuous and having intervals where paths are decreasing. The strictly increasing property of sample paths of $G(t)$ also follows by using Theorem 21.3 of Sato (1999), since
$\int_{0}^{\infty}\pi(dx) = \infty$. Let $H(t)$ be the right continuous inverse of $G(t)$, defined by
\beq
H(t) = \inf\{u>0: G(u)>t\},~ t\geq 0.
\eeq
Since the sample paths of $G(t)$ are strictly increasing with jumps, the sample paths of $H(t)$ are continuous and are constant over the intervals where $G(t)$ have jumps.
In this paper, we find out the Laplace transform of the density of the process $H(t)$ with respect to time variable and then invert
it to
get the corresponding density function $h(x,t)$. Note also that, for $x>0$,
\begin{equation}\label{feig-supBM}
\{H(t) \leq x\} = \{G(x) \geq t\} = \{\sup_{s\leq t} (B(s) + \gamma s) \leq \delta x\} = \{\delta^{-1}\sup_{s\leq t}( B(s) + \gamma s) \leq  x\},
\end{equation}
and hence $H(t) \stackrel{d} =\displaystyle\delta^{-1}\sup_{s\leq t} (B(s) + \gamma s)$, where the symbol $\stackrel{d}=$ denotes the equality in distributions. Thus, density function of $H(t)$ can also be viewed as density of supremum of Brownian motion with drift.
 It is shown that the process $H(t)$ is not a L\'evy process. The distribution of $H(t)$ is not infinitely divisible and its tail probability $\P(H(t)>x)$ decays exponentially. The density function of $h(x,t)$ solves a pseudo fractional differential equation. The PDE satisfied by the subordination of $H(t)$ with the standard Brownian motion $B(t)$ is also obtained.
\setcounter{equation}{0}
\section{First-Exit Times of an Inverse Gaussian Process}
Let $\mathcal{L}_t w(x,t) = \int_{0}^{\infty}e^{-st}w(x,t)dt$ be the Laplace transform (LT) of the function $w$ with respect the the time variable $t$.
 For a driftless subordinator $D(u)$ with corresponding L\'evy measure $\pi_D$ and density function $f$ from L\'evy-Khinchin representation, we have (see e.g. Bertoin (1996); Sato (1999))
 \beq
 \int_{0}^{\infty}e^{-st}f_{D(x)}(t)dt = e^{-x\Psi_D(s)},
 \eeq
 where
 \beq\label{lsym}
 \Psi_D(s) = \int_{0}^{\infty}(1-e^{-su})\pi_D(du)
 \eeq
 is  the Laplace symbol. First we need the following result which is a weaker version of Theorem 3.1 of Meerschaert and Scheffler (2008), and also provide a slightly different proof.

\begin{theorem}
 Let $Y= \{Y(t)\}_{t>0}$ be a strictly increasing subordinator with density  $p(x,t)$ admitting continuous partial derivatives.  Let $W(t) = \inf\{x>0: Y(x)>t\}$ so that $\{W(t)\}$  represents the first-exit time process of $\{Y(t)\}.$Then the density function $q(x,t)$ of $W(t)$  is given by
 \beq \label{mmr-result}
 q(x,t) = \int_{0}^{t}\pi_Y(t-y, \infty) p(y, x)dy, ~t>0,~x>0,
 \eeq
where $\pi_Y$ is the Levy measure with $\pi_Y (0, \infty)= \infty.$
 \end{theorem}
 \noindent\begin{proof}
 Since the sample paths of $\{Y(t)\}$ are
 strictly increasing, we have $\P(W(t)\leq x) = \P(Y(x)\geq t), x>0, t>0.$  This implies
 \begin{equation} \label{eqn2.4n}
  \int_{0}^{x}q(u,t)du = \int_{t}^{\infty}p(u,x)du, ~t>0,~x>0.
  \end{equation}
	Differentiating both sides with respect to $x$, we get
	\begin{equation} \label{eqn2.5n}
	q(x, t)= \int_{t}^{\infty} \frac{\partial}{\partial x} p(u, x) du.
	\end{equation}
	Letting now $t \rightarrow 0^{+}$, we get for $x>0$,
	\begin{equation} \label{eqn2.6n}
	q(x, 0)=  \frac{d}{dx} \int_{0}^{\infty} p(u, x) du= 0.
	\end{equation}
	Also, differentiating now \eqref{eqn2.5n} with respect to $t$, we get
	\begin{equation} \label{eqn2.7n}
	\frac{\partial}{\partial t}q(x,t) = -\frac{\partial}{\partial x}p(t, x).
	\end{equation}

 \noindent Taking the LT on both sides with respect to $t$ and 
  using $\mathcal{L}_t p(t,x) = e^{-x \Psi_Y(s)}$,  we have
 \begin{align}\label{l1}
     s\mathcal{L}_t q(x,t) - q(x,0) &= - \frac{\partial}{\partial x} \mathcal{L}_t p(t,x) = -\frac{\partial}{\partial x}
 e^{- x \Psi_Y(s)}\nonumber\\
 \implies \mathcal{L}_t[q(x,t)] &= \frac{1}{s}\Psi_Y(s)e^{- x \Psi_Y(s)} ~~\text{(using \eqref{eqn2.6n})}.
 \end{align}

\noindent Also, from (3.12) of Meerschaert and Scheffler (2008), $\mathcal{L}_t (\pi_{Y}(t, \infty))= 
\Psi_Y(s)/s$. Hence, we get
\begin{align*}\label{l2n}
\mathcal{L}_t q(x,t) &= \mathcal{L}_t \left(\pi_{Y}(t, \infty)) \mathcal{L}_t ( p(t, x) \right) \\
                     &= \mathcal{L}_t \left(\int_{0}^{t}\pi_Y(t-y, \infty) p(y, x)dy \right) \\
										 & =\mathcal{L}_t (q^{*}(x,t)) ~\text{ (say)}.
\end{align*}
\noindent Since $q(x,t)$ (see \eqref{eqn2.5n})  and $q^{*}(x,t)$ are continuous in $t$, we get  $q(x,t)=q^{*}(x,t).$
This proves the result. 
\end{proof}

\noindent Let $g(u,x)$ denote the density function of
 $G(x)\sim$ IG$(\delta x, \gamma)$ and let $h(x,t)$ denote the density function of the first-exit time process $H(t)$.
Using \eqref{l1} and the fact that $\mathcal{L}_t g(t,x) = e^{-\delta x(\sqrt{\gamma^2+2s}-\gamma)}$, we have the following result.
\begin{corollary} The LT of the density function of the first-exit time process $H(t)$ of the
inverse Gaussian process $G(t)$ is given by
 \beq\label{lthxt}
\mathcal{L}_t[h(x,t)] = \frac{\delta}{s}(\sqrt{\gamma^2+2s}-\gamma) e^{-\delta x(\sqrt{\gamma^2+2s}-\gamma)}.
\eeq
\end{corollary}

\begin{remark}
 The Laplace-Laplace transform (LLT) of the function $h$ defined by
 \bea
 \bar{h}(u,s) = \int_{\mathbb{R}}e^{-ux}\left(\int_{0}^{\infty}e^{-s t} h(x, t)dt\right) dx
 \eea
 is given by
 \beq\label{llt}
 \bar{h}(u,s) = \left(\frac{1}{s}\right)\frac{\delta(\sqrt{\gamma^2+2s}-\gamma)}{u + \delta(\sqrt{\gamma^2+2s}-\gamma)}.
 \eeq
\end{remark}
\begin{theorem}
 The density function $h(x,t)$ of $H(t)$ can be put in the following integral form
 \beq\label{density}
 h(x,t) = \frac{\delta}{\pi} e^{\delta\gamma x -\frac{\gamma^2}{2}t}\int_{0}^{\infty}\frac{e^{-ty}}{y+\gamma^2/2}
 \left[\gamma \sin(\delta x\sqrt{2y}) + \sqrt{2y}\cos(\delta x\sqrt{2y})\right] dy, \;\;x>0.
 \eeq
\end{theorem}
\noindent \begin{proof}
The density function of $H(t)$ is calculated by taking the Laplace inverse of $\mathcal{L}_t[h(x,t)]$ given in (\ref{lthxt}), by using the Laplace
inversion formula. Let $\mathcal{L}_t[h(x,t)] = F(s)$.
Then
\beq\label{f}
h(x,t) = \frac{1}{2\pi i} \int_{x_0-i\infty}^{x_0+i\infty}e^{st}F(s)ds.
\eeq
Write
\bea
F(s) = \delta e^{\delta\gamma x}\left(\frac{\sqrt{\gamma^2+2s}}{s}e^{-\delta x\sqrt{\gamma^2+2s}}-\frac{\gamma}{s}e^{-\delta x\sqrt{\gamma^2+2s}}\right) = \delta e^{\delta\gamma x}(F_2(s)-\gamma F_1(s))\ \rm{(say)}.
\eea

\begin{figure}[ht]
\centering{\includegraphics[scale=0.45]{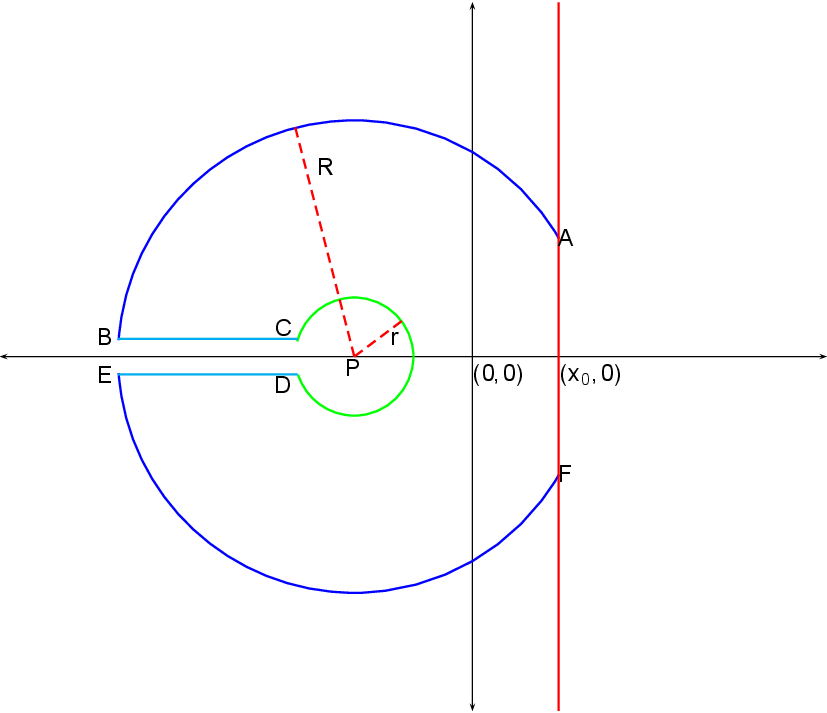}}
\caption{Contour ABCDEFA}
\end{figure}

\noindent For calculating integral in \eqref{f}, we consider a closed key-hole contour ABCDEFA with a branch point at P$=(-\gamma^2/2, 0)$ (see Figure 1).
Here AB and EF are arcs of a circle of radius $R$ with center at $P$, BC and DE are line segments parallel to negative  $x$-axis, CD is an arc $\gamma_r$ of a circle of radius $r$ with center at P and FA is the line segment from $x_0-iy$ to $x_0+iy$ with $x_0>0$. First, we calculate the above integral explicitly for
\bea
F_1(s) = \frac{1}{s}e^{-\delta x\sqrt{\gamma^2+2s}}.
\eea
By residue Theorem,
\begin{align}\label{contour}
\frac{1}{2\pi i}\int_{AB}e^{st}F_1(s) ds &+ \frac{1}{2\pi i}\int_{BC}e^{st}F_1(s) ds + \frac{1}{2\pi i}\int_{CD}e^{st}F_1(s) ds +\frac{1}{2\pi i}\int_{DE}e^{st}F_1(s) ds\nonumber\\
&+\frac{1}{2\pi i}\int_{EF}e^{st}F_1(s) ds +\frac{1}{2\pi i}\int_{x_0-iy}^{x_0+iy}e^{st}F_1(s)ds = e^{-\delta\gamma x},
\end{align}
which is the residue of $F_1$ at $s=0$.
Since $\displaystyle|F_1(s)|\leq \frac{1}{s}$, we have
\beq\label{limit}
\lim_{R\rightarrow\infty}\int_{AB}e^{st}F_1(s) ds = \lim_{R\rightarrow\infty}\int_{EF}e^{st}F_1(s) ds =0,
\eeq
(see Schiff (1999), Lemma 4.1, p. 154). Further, for the path CD, after putting $s = -\gamma^2/2 + re^{i\theta}$, we have
\begin{align}\label{cd}
\int_{CD}e^{st}F_1(s) ds &= \int_{CD}\frac{e^{st}}{s}e^{-\delta x\sqrt{\gamma^2+2s}}ds
= \int_{\pi-\epsilon}^{-\pi+\epsilon}\frac{e^{t(-\gamma^2/2+re^{i\theta})}}{-\gamma^2/2+re^{i\theta}}e^{-\delta x \sqrt{2re^{i\theta}}}
i r e^{i\theta}d\theta\nonumber\\
&\rightarrow 0,
\end{align}
as $r\rightarrow 0$, since the integrand is bounded.
\noindent Along BC, put $\gamma^2+ 2s= 2ye^{i \pi}$ so that $\sqrt{\gamma^2+2s} =i\sqrt{2y}$ and $ds =-dy$. We have
\begin{equation}\label{bc}
 \int_{BC}e^{st}F_1(s) ds = \int_{-R-\gamma^2/2}^{-r-\gamma^2/2}\frac{e^{st}}{s}e^{-\delta x\sqrt{\gamma^2+2s}}ds = e^{-t\gamma^2/2}\int_{R}^{r}\frac{1}{\gamma^2/2+y}e^{-ty-i\delta x\sqrt{2y}}dy
\end{equation}

\noindent Next along DE, take $\gamma^2+ 2s= 2ye^{-i \pi}$, this implies $\sqrt{\gamma^2+2s} =-i\sqrt{2y}$ and $ds =-dy$. This implies
\begin{equation}\label{de}
 \int_{DE}e^{st}F_1(s) ds = \int_{-r-\gamma^2/2}^{-R-\gamma^2/2}\frac{e^{st}}{s}e^{-\delta x\sqrt{\gamma^2+2s}}ds = e^{-t\gamma^2/2}\int_{r}^{R}\frac{1}{\gamma^2/2+y}e^{-ty+i\delta x\sqrt{2y}}dy.
\end{equation}
Using \eqref{bc} and \eqref{de}, we get
 \beq\label{sum}
\frac{1}{2\pi i}\int_{BC}e^{st}F_1(s) ds + \frac{1}{2\pi i}\int_{DE}e^{st}F_1(s)ds = \frac{e^{-t\gamma^2/2}}{\pi}\int_{r}^{R}\frac{e^{-ty}}{\gamma^2/2+y}\sin(\delta x\sqrt{2y})dy.
  \eeq
Using \eqref{contour}, \eqref{limit}, \eqref{cd} and \eqref{sum} with $r\rightarrow 0, R\rightarrow\infty$, we get
\beq\label{f1}
\frac{1}{2\pi i} \int_{x_0-i\infty}^{x_0+i\infty}e^{st}F_1(s)ds = e^{-\delta\gamma x} - \frac{e^{-t\gamma^2/2}}{\pi}\int_{0}^{\infty}\frac{e^{-ty}}{\gamma^2/2+y}\sin(\delta x\sqrt{2y})dy.
\eeq
Next consider
\bea
F_2(s) = \frac{1}{s}\sqrt{\gamma^2+2s}e^{-\delta x\sqrt{\gamma^2+2s}}.
\eea
Using the same procedure as for $F_1(s)$, we get for $F_2(s)$ as
\beq\label{f2}
\frac{1}{2\pi i} \int_{x_0-i\infty}^{x_0+i\infty}e^{st}F_2(s)ds = \gamma e^{-\delta\gamma x} + \frac{e^{-t\gamma^2/2}}{\pi}\int_{0}^{\infty}\frac{\sqrt{2y}e^{-ty}}{\gamma^2/2+y}\cos(\delta x\sqrt{2y})dy.
\eeq
The result follows now by using \eqref{f1} and \eqref{f2} with \eqref{f}.
 \end{proof}

\begin{remark} For every $t>0$, define $M_{\gamma}(t) = \sup_{0\leq s \leq t}(B(s) + \gamma s).$ Then for $x\geq 0$ (see Boukai (1990))
\begin{equation}\label{supBM}
\P(M_{\gamma}(t) >x) = 1 - \Phi\left(\frac{x-\gamma t}{\sqrt{t}}\right) + e^{2\gamma x}\left(1-\Phi\left(\frac{x+\gamma t}{\sqrt{t}}\right)\right),
\end{equation}
where $\Phi$ denotes cdf of standard normal distribution. Using \eqref{feig-supBM}, \eqref{density} and \eqref{supBM} are different representations of distribution function of supremum of Brownian motion with drift $M_{\gamma}(t)$. Using \eqref{mmr-result}, \eqref{lmeas} and density function of $G(t)$, an alternative representation of distribution function of $M_{\gamma}(t)$ can be written as
\begin{equation}\label{double-integral-repr}
q(x,t) = \frac{e^{\gamma x}}{2\pi}\int_{0}^{t}\left(\int_{t-y}^{\infty}u^{-3/2}e^{-\gamma^2u/2}du\right)y^{-3/2} e^{-\frac{1}{2}\left(\frac{x^2}{y} + \gamma^2 y\right)}dy. 
\end{equation}
\end{remark}
\noindent For general $\gamma>0$, it is difficult to show that \eqref{density}, \eqref{supBM} and \eqref{double-integral-repr} are equivalent. However for $\delta =1$ and $\gamma=0$, we can easily show that these three expressions are equivalent.
\begin{remark}(i)
For the particular case $\gamma =0$ and $\delta=1$, from \eqref{density}, we have
\begin{align*}
h(x,t) &= \frac{\sqrt{2}}{\pi}\int_{0}^{\infty}y^{-1/2}e^{-ty}\cos( x\sqrt{2y})dy= \frac{2\sqrt{2}}{\pi}\int_{0}^{\infty}e^{-tu^2}\cos(\sqrt{2} x u)du = \sqrt{\frac{2}{\pi t}}e^{-x^2/{2t}},
\end{align*}
by using, $\int_{0}^{\infty}e^{-tu^2}\cos(au)du = \frac{1}{2}\sqrt{\frac{\pi}{t}}e^{-a^2/{4t}}$
(see e.g. Abramowitz and Stegun (1992)).\\

\noindent From \eqref{supBM}, $\P(M_0(t)> x) = 1 - 2\Phi(x/\sqrt{t})$. Taking derivative of both sides with respect to $x$ imlies, $f_{M_0(t)}(x) = \sqrt{\frac{2}{\pi t}}e^{-x^2/{2t}}$.\\

\noindent Further, \eqref{double-integral-repr} implies
\begin{align*}
q(x,t) &= \frac{1}{2\pi} \int_{0}^{t}\left(\int_{t-y}^{\infty}u^{-3/2}du\right)y^{-3/2} e^{-\frac{x^2}{2y}}dy = \frac{x}{\pi}\int_{0}^{t}y^{-3/2}(t-y)^{-1/2}e^{-\frac{x^2}{2y}} dy\\
& = \frac{x}{\pi t}e^{-\frac{x^2}{2t}} \int_{0}^{\infty} u^{-1/2} e^{-\frac{x^2}{2t}u}du~~~~  ({\rm substitute} ~~y = t/(1+u))\\
& = \sqrt{\frac{2}{\pi t}}e^{-x^2/{2t}}.
\end{align*} 
(ii) It is obvious that, when $\gamma =0$ and $\delta=1$, $H(t) \stackrel{d}=|B(t)|$. However, $H(t)$ and $|B(t)|$
are not the same processes since the sample paths of $H(t)$ are monotonically increasing while the sample paths of $|B(t)|$ are oscillatory.
For a L\'evy process $Z(t)$, $\E(Z(t)) = t\E(Z(1))$. Since
$t\E(H(1)) = t\sqrt{2/\pi} \neq \sqrt{2t/\pi} = \E(H(t))$, we see that $H(t)$ is not a L\'evy process.
\end{remark}

\noindent For $q>0$, let $M_q(t) = \E(H(t))^q$ denote the $q$-th moment of the process $H(t)$, which may be numerically evaluated for known $t$ by using the density function $h(x,t)$. However, an explicit expression can be obtained by using the LT of $M_q(t)$. Let $\tilde{M_q}(s)$ denotes the LT of $M_q(t)$. As given in Veillette and Taqqu (2010a), we have
\beq\label{moments}
\tilde{M_q}(s) = \frac{q\Gamma(1+q)}{s\left(\Psi_G(s)\right)^q},
\eeq
where $\Psi_G(s) = \delta(\sqrt{\gamma^2+2s}-\gamma)$ denotes the Laplace symbol of the process $G(t)$. Now we have the following result for the mean and variance of the process $H(t)$. Let ${\rm Erf}(z) = \frac{2}{\sqrt{\pi}}\int_{0}^{z}e^{-u^2}du$ denotes the error function.

\begin{proposition}
The mean function $M_1(t)$ of $H(t)$ is given by
\bea
M_1(t)=\frac{1}{\delta}\sqrt{\frac{t}{2\pi}}e^{-t\gamma^2/2}+\frac{1}{2\delta\gamma}{\rm Erf}\left(\gamma\sqrt{{t}/{2}}\right)
+\frac{\gamma t}{2\delta}\left(1+{\rm Erf}\left(\gamma\sqrt{{t}/{2}}\right)\right),~ \delta,\gamma>0.
\eea
 The second moment $M_2(t)$ of $H(t)$ is
\begin{equation*}
M_2(t) = \frac{\gamma^2t^2}{\delta^2}+ \frac{2t}{\delta^2}+\frac{1}{\delta^2\gamma^2\sqrt{\pi}}\left[e^{-\gamma^2t/2}(\gamma\sqrt{2t}+\gamma^3\sqrt{2}t^{3/2}) + \sqrt{\pi}(2\gamma^2t+\gamma^4t^2-1){\rm Erf}(\gamma\sqrt{t/2})\right].
\end{equation*}
When $\gamma=0$ and $\delta=1$, we have $M_1(t) = \sqrt{2t/\pi}$ and $M_2(t)= 2t$.
\end{proposition}
\noindent\begin{proof} From \eqref{moments},
\beq\label{eq2.24}
\tilde{M_1}(s)= \frac{1}{\delta s(\sqrt{\gamma^2+2s}-\gamma)} =\frac{1}{2\delta}\frac{\sqrt{\gamma^2+2s}}{s^2} + \frac{\gamma}{2\delta}\frac{1}{s^2}.
\eeq
Using the fact(see Roberts and Kaufman (1966), p. 210)
\beq\label{Laplace}
\mathcal{L}^{-1}\left[\frac{\sqrt{s+a}}{s}\right] = \frac{e^{-at}}{\sqrt{\pi t}} +\sqrt{a}\mbox{Erf}(\sqrt{at})
\;\; {\rm and}\;\;
\int_{0}^{z}\mbox{Erf}(y)dy = z\mbox{Erf}(z)+ \frac{e^{-z^2}}{\sqrt{\pi}},
\eeq
we have
\begin{align}\label{eq2.25}
\mathcal{L}^{-1}\left[\frac{\sqrt{s+\gamma^2/2}}{\sqrt{2}\delta s^2}\right] &= \frac{1}{\sqrt{2}\delta}\left[\int_{0}^{t}\frac{e^{-\gamma^2w/2}}{\sqrt{\pi w}} dw+\frac{\gamma}{\sqrt{2}}\int_{0}^{t}\mbox{Erf}(\gamma\sqrt{w}/\sqrt{2})dw\right]\nonumber\\
&= \frac{1}{\delta}\sqrt{\frac{t}{2\pi}}e^{-t\gamma^2/2}+\frac{1}{2\delta\gamma}\mbox{Erf}\left(\gamma\sqrt{{t}/{2}}\right)
+\frac{\gamma t}{2\delta}\mbox{Erf}\left(\gamma\sqrt{{t}/{2}}\right).
\end{align}
Also, $\mathcal{L}^{-1}(1/s^2) = t$. Now the result follows for $M_1(t)$ by using \eqref{eq2.24} with \eqref{eq2.25}.\\
Further, for $M_2(t)$ we have from \eqref{moments}
\begin{align*}
\tilde{M_2}(s)= \frac{1}{\delta^2 s(\sqrt{\gamma^2+2s}-\gamma)^2}
= \frac{1}{\delta^2}\left[\frac{2\gamma^2}{s^3}+\frac{2}{s^2}+2\gamma\frac{\sqrt{\gamma^2+2s}}{s^3}\right].
\end{align*}
Using similar steps as in \eqref{eq2.25}, we obtain
\begin{align}\label{}
\mathcal{L}^{-1}\left[\frac{\sqrt{s+\gamma^2/2}}{s^3}\right] &= \int_{0}^{t}\int_{0}^{\tau}\frac{e^{-\gamma^2w/2}}{\sqrt{\pi w}} dw d\tau+\frac{\gamma}{\sqrt{2}}\int_{0}^{t}\int_{0}^{\tau}\mbox{Erf}(\gamma\sqrt{w}/\sqrt{2})dw d\tau\nonumber\\
&= \frac{1}{2\gamma^3\sqrt{2\pi}}\left[e^{-\gamma^2t/2}(\gamma\sqrt{2t}+\gamma^3\sqrt{2}t^{3/2}) + \sqrt{\pi}(2\gamma^2t+\gamma^4t^2-1)\mbox{Erf}(\gamma\sqrt{t/2})\right].
\end{align}
The result for $M_2(t)$ now follows using the simple fact that $\mathcal{L}^{-1}(1/s^3) = t^2/2$.\\
\noindent When $\gamma=0$ and $\delta=1$, we have $\tilde{M_1}(s) = s^{-3/2}/\sqrt{2}$ and $\tilde{M_2}(s) = 2/s^2$ which imply, $M_1(t) = \sqrt{2t/\pi}$ and $M_2(t)=2t$ respectively.
  \end{proof}
\noindent The following result presents the asymptotic behavior of $M_1(t) $ and Var$(H(t))$.
\begin{corollary}
As $t\rightarrow\infty$, $M_1(t) \sim \gamma t/{\delta}$ for $\gamma>0$, $ M_1(t) \sim\sqrt{2t}/(\delta\sqrt{\pi})$ for $\gamma=0$ and {\rm Var}$(H(t))\sim \gamma^2 t^2/{\delta^2}.$ Further, as $t\rightarrow 0$, $M_1(t)\sim \sqrt{2t}/(\delta\sqrt{\pi})$ and {\rm Var}$(H(t))= o(t^{1/2}).$

\end{corollary}
\noindent\begin{proof}
Using the fact that $\displaystyle\lim_{z\rightarrow\infty}{\rm Erf}(z) = 1$, we have
\bea
\lim_{t\rightarrow\infty}\frac{M(t)}{t} = \lim_{t\rightarrow\infty}\left[\frac{1}{\delta\sqrt{2\pi}}\frac{e^{-t\gamma^2/2}}{\sqrt{t}}+\frac{1}{2\delta\gamma}
\frac{1}{t}\mbox{Erf}\left(\gamma\sqrt{{t}/{2}}\right) + \frac{\gamma}{2\delta}\left(1+\mbox{Erf}
\left(\gamma\sqrt{{t}/{2}}\right)\right)\right] = \frac{\gamma}{\delta}.
\eea
Further, using the fact $\displaystyle\lim_{z\rightarrow 0}$Erf$(z) =0$, we get
\begin{align*}
\lim_{t\rightarrow 0} \frac{M(t)}{\sqrt{t}} &= \lim_{t\rightarrow 0}\left[\frac{1}{\delta\sqrt{2\pi}}e^{-t\gamma^2/2}+\frac{1}{2\delta\gamma}\frac{1}{\sqrt{t}}\mbox{Erf}
\left(\gamma\sqrt{{t}/{2}}\right) + \frac{\gamma}{2\delta}\sqrt{t}\left(1+\mbox{Erf}
\left(\gamma\sqrt{{t}/{2}}\right)\right)\right]\\
& = \frac{1}{\delta\sqrt{2\pi}}+ \frac{1}{2\delta\gamma}\lim_{t\rightarrow 0}\frac{1}{\sqrt{t}}\mbox{Erf}
\left(\gamma\sqrt{{t}/{2}}\right) = \frac{1}{\delta\sqrt{2\pi}}+\frac{1}{\delta\sqrt{2\pi}}
= \frac{1}{\delta}\sqrt{\frac{2}{\pi}}.
\end{align*}
The result for asymptotic behavior of variance of $H(t)$ also follows similarly.
\end{proof}

\noindent Next, we look at the tail behavior of the distribution of $h(x,t)$, which is useful for studying the infinite divisibility.
\begin{theorem}
 The tail probability $\P(H(t)>x)$ decays exponentially, that is,
  \beq\label{tail}
\P(H(t)>x) = O(x^{-1}e^{\delta\gamma x-x^2/{4t}}),\ \ {\rm as} \ x\rightarrow\infty.
\eeq
\end{theorem}
\noindent\begin{proof} We have from \eqref{density},
\begin{align}\label{asymptotic}
h(x,t) &= \frac{\delta}{\pi} e^{\delta\gamma x -\frac{\gamma^2}{2}t}\left[2\gamma\int_{0}^{\infty}\frac{\omega e^{-t\omega^2}}{\omega^2+\gamma^2/2}\sin(\delta\sqrt{2}x\omega) d\omega +2\sqrt{2}
 \int_{0}^{\infty}\frac{\omega^2 e^{-t\omega^2}}{\omega^2+\gamma^2/2}\cos(\delta\sqrt{2}x\omega)d\omega\right],
 \end{align}
 obtained by substituting $\omega =\sqrt{y}$.
For a function $q$ such that $q(z)$ is real for real $z$ and $q(z)= O(z^{-\eta})$, $\eta>0$, as Re$(z)\rightarrow\infty$
 (see e.g. Olver(1974), p. 79)
\bea
\int_{\mathbb{R}}\frac{ze^{ixz}}{z^2+\alpha^2}q(z)dz = \pi i e^{-\alpha x}q(i\alpha)+\epsilon(x),
~~{\rm where}~~
\epsilon(x) = \int_{\mathcal{L}}\frac{ze^{ixz}}{z^2+\alpha^2}q(z)dz
\eea
with the line $\mathcal{L}$: $z=y +i w$, $w >\alpha$ and $-\infty<y<\infty$. Taking $q(z)= e^{-tz^2}$ and the corresponding $\epsilon_1(x)$,
\bea
\int_{\mathbb{R}}\frac{ze^{ixz}}{z^2+\alpha^2} e^{-tz^2}dz = \pi i e^{-\alpha x}e^{t\alpha^2}+\epsilon_1(x).
\eea
Comparing the imaginary parts in both sides gives
\begin{align}\label{asy1}
\int_{\mathbb{R}}\frac{z \sin(xz)}{z^2+\alpha^2} e^{-tz^2}dz &= \pi e^{-\alpha x+t\alpha^2}+\epsilon_1(x)\nonumber\\
\implies \int_{0}^{\infty}\frac{z \sin(xz)}{z^2+\alpha^2} e^{-tz^2}dz& = \frac{\pi}{2}e^{-\alpha x+t\alpha^2}+\epsilon_1(x).
\end{align}
Similarity, for $q(z)= ze^{-tz^2}$, we get
\beq\label{asy2}
\int_{0}^{\infty}\frac{z^2\cos(xz)}{z^2+\alpha^2} e^{-tz^2}dz = -\frac{\pi}{2}\alpha e^{-\alpha x}e^{t\alpha^2}+\epsilon_2(x).
\eeq
Also, $\epsilon_1(x) = \epsilon_2(x)= O(x^{-1}e^{-x^2/{4t}})$, as $x\rightarrow\infty$ (see Exercise 6.2 of Olver (1974)).
Now using \eqref{asy1} and \eqref{asy2} with \eqref{asymptotic}, we have
\begin{align}\label{eq2.41}
h(x,t) &= \frac{\delta}{\pi} (2\gamma+2\sqrt{2})e^{\delta\gamma x -\frac{\gamma^2}{2}t}\epsilon_1(x)
= O(x^{-1}e^{\delta\gamma x-x^2/{4t}}), \ \ \rm{as}\ x\rightarrow\infty.
\end{align}
Now
\begin{align*}
\displaystyle\lim_{x\rightarrow\infty}\Big|\frac{\P(H(t)>x)}{x^{-1}e^{\delta\gamma x-x^2/{4t}}}\Big| &=
\lim_{x\rightarrow\infty}\Big|\frac{-h(x,t)}{\frac{d}{dx}(x^{-1}e^{\delta\gamma x-x^2/{4t}})}\Big|~~\mbox{(using L'H$\hat{o}$pital's rule)}\\
&\leq M, ~~M>0,
\end{align*}
which proves the  result.
\end{proof}

\noindent Since the tail probability $\P(H(t)>x)$ decays exponentially, moments of all orders exit for $H(t)$. 

\begin{proposition} \label{dlim}
 The density function $h(x,t)$ satisfies
 \bea
 \lim_{x\rightarrow 0}h(x,t) = h(0,t) = \delta e^{-\gamma^2/2}\left(\sqrt{\frac{2}{\pi t}}-\gamma e^{\gamma^2t/2}{\rm Erfc}\left(\frac{\gamma\sqrt{t}}{\sqrt{2}}\right)\right).
 \eea
  Further, $\lim_{x\rightarrow 0}h_x(x,t) =h_x(0,t) = 2\delta\gamma h(0,t)$
 and $h_x(x,t)=O(N e^{\delta\gamma x-x^2/{4t}})$, as $x\rightarrow\infty$, with $N>0$.
\end{proposition}
\noindent\begin{proof} Note that,
$\displaystyle h(x,t) = \frac{\delta}{\pi}e^{\delta\gamma x-\gamma^2/2}\int_{0}^{\infty}I(t;x,y)dy$,
where $I(t;x,y)$ follows from \eqref{density}. Then
\begin{equation}\label{sin}
\begin{split}
 |I(t;x,y)|&=\Big|\frac{e^{-ty}}{y+\gamma^2/2}[\gamma \sin(\delta x\sqrt{2y})+\sqrt{2y}\cos(\delta x\sqrt{2y})]\Big|\\
 &= \Big|\frac{e^{-ty}}{y+\gamma^2/2}\sqrt{2y+\gamma^2} \sin(\delta x\sqrt{2y}+\theta)\Big|
 \leq \frac{2e^{-ty}}{\sqrt{2y+\gamma^2}},
 \end{split}
\end{equation}
where $\theta = \sin^{-1}\left(\sqrt{2y/(\gamma^2+2y)}\right)$. By dominated convergence theorem, we have
\begin{align*}
\lim_{x\rightarrow 0}h(x,t)& = \frac{\delta}{\pi}e^{-\gamma^2/2}\int_{0}^{\infty}I(t;0,y)dy
= \frac{2\sqrt{2}\delta}{\pi}e^{-\gamma^2/2}\int_{0}^{\infty}\frac{\sqrt{y}e^{-ty}}{2y+\gamma^2}dy\\
&= \frac{\delta}{\pi\sqrt{t}}e^{-\gamma^2/2}\left(\sqrt{2\pi}-\gamma \pi \sqrt{t}e^{\gamma^2t/2}{\rm Erfc}\left(\frac{\gamma\sqrt{t}}{\sqrt{2}}\right)\right)\\
&= \delta e^{-\gamma^2/2}\left(\sqrt{\frac{2}{\pi t}}-\gamma e^{\gamma^2t/2}{\rm Erfc}\left(\frac{\gamma\sqrt{t}}{\sqrt{2}}\right)\right).
\end{align*}
Note that $I(t;x,y)$ and
$
\frac{\partial}{\partial x} (I(t;x,y)) = e^{-ty}\left(\delta\gamma\sqrt{2y}\cos(\delta x\sqrt{2y})-2\delta\gamma \sin(\delta x \sqrt{2y})\right)/({y+\gamma^2/2})
$
are continuous in $x$ and $y$. Following the steps in \eqref{sin}, we have
$|\frac{\partial}{\partial x} I(t;x,y)|\leq{\delta\sqrt{2y}e^{-ty}}/{(y+\gamma^2/2)}$ which is independent of $x$
and $\int_{0}^{\infty}\frac{2\sqrt{2}\delta\sqrt{y}e^{-ty}}{2y+\gamma^2}dy<\infty$.
Hence, $\int_{0}^{\infty}\frac{\partial}{\partial x} I(t;x,y)dy$ is uniformly convergent and so we can differentiate under the integral sign. Thus,
\beq\label{derivative}
h_x(x,t) = \frac{\delta}{\pi}e^{\delta\gamma x-\gamma^2/2}\left[\delta\gamma\int_{0}^{\infty}I(t;x,y)dy+\int_{0}^{\infty}\frac{\partial}{\partial x} I(t;x,y)dy\right]
\eeq
and
\begin{equation*}
 h_x(0,t) = \frac{4\sqrt{2}\delta^2\gamma}{\pi}e^{-\gamma^2/2}\int_{0}^{\infty}\frac{\sqrt{y}e^{-ty}}{2y+\gamma^2}dy= 2\delta\gamma h(0,t).
\end{equation*}
Using \eqref{eq2.41} and an application of  L'H$\hat{o}$pital's rule, we can show that, as $x\rightarrow\infty,$
$h_x(x,t)=O(N e^{\delta\gamma x-x^2/{4t}}),$ for some $N>0$.
\end{proof}

\noindent A consequence of Theorem 2.3 is the following result known in the literature.
\begin{proposition}\label{prop2.2}
The distribution of $H(t)$ is not infinitely divisible.
\end{proposition}
\noindent \begin{proof} Since $\P(H(t)>x) \leq Mx^{-1}e^{\delta\gamma x-x^2/{4t}}$, for large $x$ and
some $M>0$, we have
\begin{align*}
 \lim_{x\rightarrow\infty}\left(\frac{-\ln \P(H(t)>x)}{x \ln x}\right)& \geq \lim_{x\rightarrow\infty}
 \left(\frac{-\ln(\frac{M}{x}e^{\delta\gamma x-x^2/{4t}})}{x \ln x}\right)\rightarrow\infty.
\end{align*}
Hence, by Corollary 9.9 of Steutel and Van Harn (2004, p. 200), $H(t)$ is Gaussian. This is a contradiction
and hence the result follows.
\end{proof}
\noindent Note that when $\gamma=0$, $h(x,t)$ is a folded Gaussian density which is also not infinitely
divisible, a known result (see e.g. p. 126 of Steutel and Van Harn (2004)). The above result can be extended to the first-exit time of a $\beta$-stable subordinator.

\begin{proposition}
The tail probability of the first-exit time process $E(t)$ of the $\beta$-stable subordinator $D(t)$ satisfies
 \begin{equation*}
 \displaystyle \P(E(t)>x) = O\left(\frac{M(1-\beta)}{N}x^{\frac{(2-\beta/2)}{(\beta-1)}-\frac{1}{\beta}}e^{-Nx^{\frac{1}{(1-\beta)}}}\right), ~~\mbox{as}~x\rightarrow\infty.
 \end{equation*}
\end{proposition}
\noindent\begin{proof} Let $f(x,1)$ denote the density
function of a $\beta$-stable $(0<\beta< 1)$ random variable
$D(1)$. Then $m(x,t) =
\frac{t}{\beta}x^{-1-1/{\beta}}f(tx^{-1/{\beta}}, 1), ~x>0,$
is the density function of the first-exit time process $E(t)$ of
the process $D(t)$, defined by $E(t) = \inf\{x>0:
D(x)>t\}$ (see e.g. Meerschaert and Scheffler (2004)). Then, from (Uchaikin and Zolotarev (1999))
\bea
 f(x,1)\leq
Kx^{(1-\beta/2)/(\beta-1)}\exp(-(1-\beta)(x/\beta)^{\beta/(\beta-1)}),
\eea for $x$ sufficiently small. Hence, as $x\rightarrow\infty,$
 $m(x,t)\leq
K(tx^{-1/\beta})^{(1-\beta/2)/(\beta-1)}\exp(-(1-\beta)(tx^{-1/{\beta}}/\beta)^{\beta/(\beta-1)}).$
In other words \beq
m(x,t)=O\left(Mx^{(1-\beta/2)/(\beta-1)-1-1/(\beta)}
\exp(-Nx^{1/(1-\beta)})\right),
\eeq where $M=(K/{\beta})t^{(1-\beta/2)/(\beta-1)+1}$ and
$N=(1-\beta)(t/\beta)^{\beta/{(\beta-1})}$. Thus, we obtain
\begin{align*}
&\lim_{x\rightarrow\infty}\Bigg|\frac{\P(E(t)>x)}{(M(1-\beta)/N)x^{(2-\beta/2)/(\beta-1)-1/\beta}\exp(-Nx^{1/(1-\beta)})}\Bigg|
~~\left(\left(\frac{0}{0}\right)~\mbox{form}\right)\\
&=\lim_{x\rightarrow\infty}\Bigg|\frac{-f(x,t)}{(M(1-\beta)/N)\frac{d}{dx}\left(x^{(2-\beta/2)/(\beta-1)-1/\beta}\exp(-Nx^{1/(1-\beta)})\right)}\Bigg| =1.
\end{align*}
The result follows now.
\end{proof}
\noindent As application  of the above result is the following.
\begin{proposition}\label{id-hitting-time}
 The distribution function of $E(t)$ is not infinitely divisible.
\end{proposition}
\begin{proof} Since $$\P(E(t)>x)\leq
\frac{M(1-\beta)}{N}x^{(2-\beta/2)/(\beta-1)-1/\beta}\exp(-Nx^{1/(1-\beta)})$$
for large $x$, we have \beq \ln \P(E(t)>x)\leq
\ln(M(1-\beta))-\ln N + ((2-\beta/2)/(\beta-1)-1/\beta) \ln x
-Nx^{1/(1-\beta)}. \eeq

\noindent Since, $0<\beta<1$,
\begin{align*}
\lim_{x\rightarrow\infty}\left(\frac{-\ln \P(E(t)>x)}{x \ln x}\right)\geq \lim_{x\rightarrow\infty}\left(\frac{Nx^{1/(1-\beta)}}{x \ln x}\right)
= \lim_{x\rightarrow\infty}\frac{Nx^{\beta/(1-\beta)}}{\ln x} = \infty,
\end{align*}
and hence by using a similar argument as in Proposition \ref{prop2.2}, $E(t)$ is not infinitely divisible.
\end{proof}

\begin{remark}
Let $D_1(t), D_2(t), \cdots, D_n(t)$ be $n$ independent standard stable processes with indices $\beta_1, \beta_2,\cdots, \beta_n$ respectively and let  $E_1(t), E_2(t), \cdots, E_n(t)$ be the corresponding inverse stable processes. Now $\E e^{-sD_1(D_2(t))} = \E e^{-D_2(t)
s^{\beta_1}}= e^{-ts^{\beta_1\beta_2}}$, and hence $D^{*2}(t)= D_1\circ D_2(t) = D_1(D_2(t))$ is also stable with index $\beta_1\beta_2.$ More generally, the composition $D^{*n}(t)=D_1\circ D_2\circ \cdots \circ D_n(t)$ is also stable with index $\beta_1\beta_2\cdots\beta_n$. Let $h^*(x,t)$ be the density function of 
$E^*(t) = \inf\{s>0: D_1(D_2(s))>t\}$. Then using \eqref{l1}, $\mathcal{L}_t(h^*(x,t)) = s^{\beta_1\beta_2-1}e^{-xs^{\beta_1\beta_2}}$. Also, let $h(x,t)$ be the density function of $E_1(E_2(t))$. Let $h_1(x,t)$ and $h_2(x,t)$ be the densities of $E_1(t)$ and $E_2(t)$ respectively. Then, with $u=s^{\beta_2}$,
\begin{align*}
\mathcal{L}_t(h(x,t)) &= \int_{0}^{\infty}h_1(x,r)\mathcal{L}_t(h_2(r,t))dr = \int_{0}^{\infty}h_1(x,r)s^{\beta_2-1}e^{-ur}dr\\
& = s^{\beta_2-1}\int_{0}^{\infty}h_1(x,r)e^{-ur}dr \\
& =s^{\beta_2-1}e^{-xu^{\beta_1}}u^{\beta_1-1} = s^{\beta_1\beta_2-1}e^{-xs^{\beta_1\beta_2}}.		
\end{align*} 
This implies $h^{*}(x,t) = h(x,t)$ or $E^{*}(t) \stackrel{d}= E_1(E_2(t))$.
Let $E^{**}(t) = \inf\{s>0: D^{*n}(s)>t\}$ be the first-exit time of the composition $D^{*n}(t)$. Then $E^{**}(t) \stackrel{d}= E^{*n}(t)$. 
Using a similar argument as in Proposition \ref{id-hitting-time}, it follows that the process $E^{*n}(t)\stackrel{d}=E^{**}(t)$  is also not infinitely divisible.
\end{remark}

\noindent The sample paths of IG process and its first-exit time process $H(t)$ are generated by using the algorithm given by Michael, Schucany and Hass (see Cont and Tankov (2004), p. 182 or Devroye (1986)) and are presented in Figure 2.

\begin{figure}[h]
\centering{\includegraphics[width=.7\textwidth, height =5.5cm]{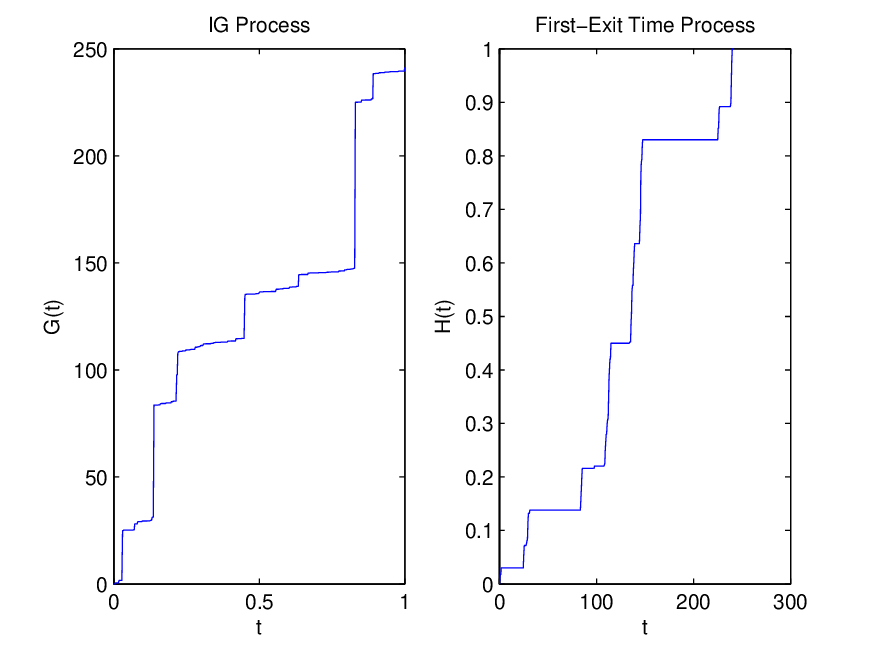}}
\caption{ \it Sample paths of IG process $G(t)$ for $\delta = \gamma=1$ and the corresponding first-exit time process $H(t)$.}
\end{figure}

\setcounter{equation}{0}
\section{The PDE Connection}

In this section, we obtain the PDE satisfied by the density function $h(x,t)$. Let us, for convenience, denote by $\tilde{h}(x,s)$ and
$\hat{h}(u,t)$  the LT and FT of $h$ respectively. Now we have the following result.
\begin{theorem} \label{thm3.1}
The densities $h(x,t)$ of the process $H(t)$ solve
\beq\label{pde}
\frac{\partial^2 h}{\partial x^2} - 2\delta\gamma \frac{\partial h}{\partial x} =  2 \delta^2 \frac{\partial h}{\partial t}
+2\delta^2h(x,0)\delta_{0}(t),
\eeq
where $\delta_0(t)$ is the Dirac's delta function.
\end{theorem}
\noindent\begin{proof} We have from \eqref{lthxt}
\bea
\tilde{h}(x,s) = \frac{\delta}{s}(\sqrt{\gamma^2+2s}-\gamma) e^{-\delta x(\sqrt{\gamma^2+2s}-\gamma)}.
\eea
This implies
\beq\label{eq33}
\frac{\partial}{\partial x}\tilde{h}(x,s) = -\delta(\sqrt{\gamma^2+2s}-\gamma)\tilde{h}(x,s),
\;{\mbox and}\;
\frac{\partial^2}{\partial x^2}\tilde{h}(x,s) = \delta^2(\sqrt{\gamma^2+2s}-\gamma)^2\tilde{h}(x,s).
\eeq
Using \eqref{eq33}, we get
\bea
\left(\frac{\partial^2}{\partial x^2} - 2\delta\gamma \frac{\partial}{\partial x}\right)\tilde{h}(x,s) = 2\delta^2(s \tilde{h}(x,s)-h(x,0)) + 2\delta^2h(x,0).
\eea
After inverting the above LT and using $\mathcal{L}_t(\delta_0(t)) =1)$, we get
\bea
\frac{\partial^2 h}{\partial x^2} - 2\delta\gamma \frac{\partial h}{\partial x} =  2 \delta^2 \frac{\partial h}{\partial t}+
2\delta^2h(x,0)\delta_0(t).
\eea
\end{proof}

\begin{remark}
The density function $g(u,t)$ of $G(t)$ satisfies the following PDE (see Kumar {\it et. al.} (2011))
\beq\label{igpde}
\frac{\partial^2 g}{\partial t^2} - 2\delta\gamma \frac{\partial g}{\partial t} =  2 \delta^2 \frac{\partial g}{\partial u}.
\eeq
Note that \eqref{pde} has a close connection with equation \eqref{igpde}.
\end{remark}

\noindent The result in Theorem \ref{thm3.1} can be generalized to the first-exit times of tempered stable subordinators.
It is well known that for a stable distribution with index $0<\beta<1$, the $p$-th moment is infinite for $p\geq \beta$, and its tail probability decays polynomially (see e.g. Samorodnitsky and Taqqu (2000)). To overcome this limitation, tempered stable (TS) distributions are introduced, which are infinitely divisible and hence correspond to a L\'evy process. Let $f(x,t)$ denote the density function of a $\beta$-stable ($0<\beta<1$) process $D(t)$ with
LT
\bea
\int_{0}^{\infty}e^{-sx}f(x,t)dx = e^{-ts^{\beta}}.
\eea

\noindent A tempered stable subordinator $D_{\mu}(t)$ has a density $f_{\mu}(x,t)= e^{-\mu x+\mu^{\beta}t}f(x,t),~~ \mu>0$. The L\'evy measure corresponding to a tempered stable process is $\pi_{D_{\mu}}(x) = cx^{-\beta-1}e^{-\mu x},~c>0, ~x>0$ (see e.g. Cont and Tankov (2004, p. 115)).
Since $\int_{0}^{\infty}\pi_{D_{\mu}}(x)dx = \infty$, Theorem 21.3 of Sato (1999) implies that the sample paths of $D_{\mu}(t)$ are strictly increasing, since jumping times are dense in $(0,\infty)$. Further
\beq\label{ts-lt}
\int_{0}^{\infty}e^{-sx}f_{\mu}(x,t)dx = e^{-t((s+\mu)^{\beta}-\mu^{\beta})}.
\eeq
Next we have the following result. For similar results in other contexts, see Nane (2008).
\begin{proposition} For $\beta =\frac{1}{n}, ~n\geq 2$,
the density function $m_{\mu}(x,t)$ of the first-exit time process of $D_{\mu}(t)$ satisfies the following PDE:
\bea
\sum_{j=1}^{n}(-1)^j {\binom{n}{j}} \mu^{(1-j/n)}\frac{\partial^j}{\partial x^j}m_{\mu}(x,t) = \frac{\partial}{\partial t}m_{\mu}(x,t)+\delta_0(t)m_{\mu}(x,0).
\eea
\end{proposition}
\begin{proof}
 We prove this result by induction. Using \eqref{l1} and \eqref{ts-lt} to get
 \begin{equation*}
 \tilde{m}_{\mu}(x,s) = \frac{1}{s}((s+\mu)^{\beta}-\mu^{\beta})e^{-t((s+\mu)^{\beta}-\mu^{\beta})}.
 \end{equation*}
 For $n=2$,
 \beq\label{ts1}
 \frac{\partial}{\partial x}\tilde{m}_{\mu}(x,s) = -((s+\mu)^{1/2}-\mu^{1/2})\tilde{m}_{\mu}(x,s) \;{\mbox and}\; \frac{\partial^2}{\partial x^2}\tilde{m}_{\mu}(x,s) = ((s+\mu)^{1/2}-\mu^{1/2})^2\tilde{m}_{\mu}(x,s).
 \eeq
 Using \eqref{ts1}, we get
 \begin{align*}
 \left( \frac{\partial^2}{\partial x^2}-2\mu^{1/2}\frac{\partial}{\partial x}\right)\tilde{m}_{\mu}(x,s) =(s\tilde{m}_{\mu}(x,s)-m_{\mu}(x,0))+ m_{\mu}(x,0).
 \end{align*}
Inverting the LT to get
\bea
\left( \frac{\partial^2}{\partial x^2}-2\mu^{1/2}\frac{\partial}{\partial x}\right)m_{\mu}(x,t) = \frac{\partial}{\partial t}m_{\mu}(x,t) + m_{\mu}(x,0).
\eea
Similarly, for $n=3$
\bea
\left(\frac{\partial^3}{\partial x^3}-3\mu^{1/3}\frac{\partial^2}{\partial x^2}+ 3\mu^{2/3}\frac{\partial}{\partial x}\right)m_{\mu}(x,t) = (-1)^3\left(\frac{\partial}{\partial t}m_{\mu}(x,t) + m_{\mu}(x,0)\right).
\eea
The result now follows in a similar manner for a general $n$.
\end{proof}

\noindent We next return to the discussion of inverse Gaussian and its first-exit time processes. 
\noindent We now obtain the fractional PDE satisfied by $h(x, t)$, the density function of $H(t)$. To this end,
note that
 the Riemann-Liouville fractional derivative of order $\beta$  $(0<\beta<1)$ is defined by (see e.g. Podlubny (1999))
\beq \label{R-L}
  \frac{\partial^{\beta}}{\partial t^{\beta}}g(t)= \frac{1}{\Gamma(1-\beta)}\frac{d}{dt}\int_{0}^{t}\frac{g(s)ds}{(t-s)^{\beta}}.
  \eeq
Also,  the Caputo fractional derivative of order $0<\beta<1$ is defined by
\beq\label{Cap}
\left(\frac{\partial}{\partial t}\right)^{\beta}f(t) = \frac{1}{\Gamma(1-\beta)}\int_{0}^{t}\frac{f'(r)}{(t-r)^{\beta}}dr,
\eeq
where $f'$ denotes the derivative of $f$.  Since $L[t^{-\beta}] = s^{\beta-1}/\Gamma(1 - \beta)$, it follows that
\beq\label{Lap-RL}
\mathcal{L}_t\left[\frac{\partial^{\beta}}{\partial t^{\beta}}g(t)\right] = s^{\beta} \tilde{g} (s)
\eeq
and the Caputo fractional derivative in \eqref{Cap} has Laplace transform $s^{\beta} \tilde{f} (s)-s^{\beta-1} f(0).$
                                                                        ̃
Hence, the two derivatives in \eqref{R-L} and \eqref{Cap} are related by
 \bea
 \frac{\partial^{\beta}}{\partial t^{\beta}}g(t) = \left(\frac{\partial}{\partial t}\right)^{\beta}g(t) + \frac{g(0)t^{-\beta}}{\Gamma(1-\beta)}.
\eea
Similar to Baeumer and Meerschaert (2010), define the Riemann-Liouville tempered fractional derivative of order $0 < \beta < 1$ by
\bea
\frac{\partial^{\beta,\lambda}}{\partial t^{\beta, \lambda}}g(t) = e^{-\lambda t}\frac{1}{\Gamma(1-\beta)}\frac{d}{dt}\int_{0}^{t}\frac{e^{\lambda s}g(s)ds}{(t-s)^{\beta}} - \lambda^{\beta}g(t).
\eea
We have the following result.
\begin{proposition}
 The density function $h(x,t)$ of $H(t)$ satisfy also the following pseudo-fractional PDE, namely,
 \bea
 \partial_xh(x,t)+\sqrt{2}\delta \frac{\partial^{1/2,\lambda}}{\partial t^{1/2, \lambda}}h(x,t) = 0, 
 \eea
 with $\lambda = \gamma^2/2$.
For $\delta=1$ and $\gamma=0$, the density function $h(x,t)$ satisfies
\beq\label{particular}
\frac{\partial}{\partial x}h(x,t) +\sqrt{2}\left(\frac{\partial}{\partial t}\right)^{1/2}h(x,t) + \sqrt{\frac{2}{\pi t}}h(x,0)=0.
\eeq
\end{proposition}
\noindent \begin{proof} Using \eqref{llt}, the LLT of $h(x,t)$ satisfies
\beq\label{eq3.4}
u\bar{h}(u,s) +\delta(\sqrt{\gamma^2+2s}-\gamma)\bar{h}(u,s) = s^{-1}\delta(\sqrt{\gamma^2+2s}-\gamma).
\eeq
Since $\tilde{h}(0,s) = s^{-1}\delta(\sqrt{\gamma^2+2s}-\gamma)$, we have
\bea
u\bar{h}(u,s) - \tilde{h}(0,s) +\delta(\sqrt{\gamma^2+2s}-\gamma)\bar{h}(u,s) = 0.
\eea
Invert the LT with respect to variable $u$ to get
\beq\label{lt-s}
\partial_x\tilde{h}(x,s) +\delta(\sqrt{\gamma^2+2s}-\gamma)\tilde{h}(x,s) = 0.
\eeq
\noindent Using \eqref{R-L}, \eqref{Lap-RL} and  $\mathcal{L}_t(e^{\lambda t}g(t)) = \tilde{g}(s-\lambda)$, we get
\beq\label{temp-RL}
\mathcal{L}_t\left[\frac{\partial^{1/2,\lambda}}{\partial t^{1/2, \lambda}}g(t)\right] = \mathcal{L}_t\left[e^{-\lambda t} \frac{d^{1/2}}{dt^{1/2}}(e^{\lambda t}g(t))-\lambda^{1/2}g(t)\right] = \left((s+\lambda)^{1/2} - \lambda^{1/2}\right)\tilde{g}(s).
\eeq
Using \eqref{temp-RL}, invert the LT with respect to variable $s$ in \eqref{lt-s} to obtain
\beq\label{hig}
 \partial_xh(x,t)+\sqrt{2}\delta \frac{\partial^{1/2,\lambda}}{\partial t^{1/2, \lambda}}h(x,t) = 0,
 \eeq
with $\lambda = \gamma^2/2$.
Further, for $\delta =1$ and $\gamma=0$, we have from \eqref{eq3.4}
\bea
u\bar{h}(u,s) +\sqrt{2s}\bar{h}(u,s) = \sqrt{\frac{2}{s}},\;\; {\rm which \;implies}\;\;
u\bar{h}(u,s) - \tilde{h}(0,s) +\sqrt{2s}\bar{h}(u,s) = 0.
\eea
Invert the above LT to obtain
\bea
 \frac{\partial}{\partial x}\tilde{h}(u,s) +\sqrt{2s}\tilde{h}(u,s)=0.
\eea
Thus
\bea
 \frac{\partial}{\partial x}\tilde{h}(u,s) +\sqrt{2}\left(\sqrt{s}\tilde{h}(u,s) - \frac{h(x,0)}{\sqrt{s}}\right) + \sqrt{2}\frac{h(x,0)}{\sqrt{s}}=0.
\eea
Inverting the above LT and using (see e.g. Podlubny (1999), p. 106)
\beq\label{feq}
\mathcal{L}_t\left[\frac{\partial^{\beta}}{\partial t^{\beta}}f(x,t)\right]=s^{\beta}\mathcal{L}_t[f(x,t)]-s^{\beta-1}f(x,0),\ 0<\beta<1,
\eeq
the result follows.
\end{proof}

\begin{remark}
 The density function of $G(t)$ has FT $\hat{g}(u,t) = e^{-\delta t(\sqrt{\gamma^2+2iu}-\gamma)}$ (see e.g. Applebaum (2009), p.54). Then
 $\frac{\partial}{\partial t} \hat{g}(u,t) = -\delta (\sqrt{\gamma^2+2iu}-\gamma)\hat{g}(u,t),$
 and after inverting the FT
 \beq\label{fig}
 \partial_t g(x,t) = -\delta (\sqrt{\gamma^2+2\partial_x}-\gamma)g(x,t).
 \eeq	
In particular for $\delta =1/\sqrt{2}$ and
 $\gamma=0$, the LT of $g(x,t)$ with respect to the variable $x$ is
 $\tilde{g}(s,t)= e^{- t\sqrt{s}}.$
 Using the arguments as in Proposition 3.1 and the fact that $g(0,t) = 0$, we have
 \bea
\frac{\partial}{\partial t}g(x,t) +\left(\frac{\partial}{\partial x}\right)^{1/2}g(x,t)=0,
\eea
 which is the governing equation of $\frac{1}{2}$-stable subordinator, the analogue of inverse Gaussian for $\delta = 1/\sqrt{2}$ and $\gamma=0.$ 
\end{remark}

\noindent Finally, we look at certain subordinated process. Note the subordinated processes have interesting connections to PDE (see e.g. Allouba (2002); Allouba and Zheng (2001); Baeumer {\it et.\ al.} (2009), Nane (2008)). We consider here the subordinated process $X(t) = B(H(t))$, where $H(t)$ is the first-exit time of the inverse Gaussian process and $B(t)$ is the standard Brownian motion. Let $u(x, t)$ denotes the density function of the process $X(t)$.
Let $h_x$ denotes the partial derivative of $h$ with respect to $x$.

\begin{theorem}\label{th4.2}
The densities $u(x,t)$ of the subordinated process $X(t)$ satisfies the following PDE
 \begin{align}
2\delta^2\frac{\partial }{\partial t}u(x,t) & = \Big(\frac{1}{4}\frac{\partial^4}{\partial x^4} + \delta\gamma\frac{\partial^2}{\partial x^2}\Big)u(x,t) + f_t(x,0)h(0,t)-2\delta^2u(x,0)\delta_0(t).
\end{align}
\end{theorem}
\noindent\begin{proof}
Since the inner and outer process are independent, the density function of $X(t)$ is
\beq
u(x, t) = \int_{0}^{\infty} w(x, r) h(r, t)\,dr,
\eeq
where $w(x, t)$ is the density of the Brownian motion $B(t)$, and $h(r,t)$ is the density of the process $H(t)$.
 Since the FT
$
\hat{w} (u, t) = \exp(-\frac{1}{2} tu^2),
$
we have
\beq\label{bmge}
\frac{\partial \hat w}{\partial t} = \frac{1}{2} (iu)^2 \hat w, \;\; {\rm and \; hence}\;\; \frac{\partial w}{\partial t} = \frac{1}{2}\frac{\partial^2 w}{\partial x^2},
\eeq
which is the governing equation for density of $B(t)$, the standard heat equation.  Now write
\begin{align*}
2\delta^2 \frac{\partial}{\partial t} u(x, t)
& = \int_{0}^{\infty} w(x,r)2\delta^2\frac{\partial}{\partial t}h(r, t)dr\\
& =  \int_{0}^{\infty}  w(x, r)\Big(\frac{\partial^2 }{\partial r^2} - 2\delta\gamma \frac{\partial }{\partial r}\Big)h(r, t)dr-2\delta^2\delta_0(t)\int_{0}^{\infty}w(x,r)h(r,0)dr\\
&= w(x,r)\frac{\partial}{\partial r}h(r,t)\Big|_{r=0}^{\infty}- \frac{\partial}{\partial r}w(x,r)h(r,t)\Big |_{r=0}^{\infty}-2\delta\gamma w(x,r)h(r,t)\Big|_{r=0}^{\infty}\\
&\hspace{1cm}+\int_{0}^{\infty} \frac{\partial^2}{\partial r^2}w(x, r) h(r, t)dr + 2\delta\gamma \int_{0}^{\infty} \frac{\partial}{\partial r}w(x, r) h(r, t)dr - 2\delta^2 u(x,0)\delta_0(t)\\
& = -w(x,0)h_x(0,t)+ w_t(x,0)h(0,t)+2\delta\gamma w(x,0)h(0,t)\\
&\hspace{1cm}+ \frac{1}{4}\int_{0}^{\infty} \frac{\partial^4}{\partial x^4}w(x, r) h(r, t)dr + \delta\gamma \int_{0}^{\infty} \frac{\partial^2}{\partial x^2}w(x, r) h(r, t)dr - 2\delta^2 u(x,0)\delta_0(t)\\
&= -w(x,0)h_x(0,t)+ w_t(x,0)h(0,t)+2\delta\gamma w(x,0)h(0,t)\\
&\hspace{1cm}+ \Big(\frac{1}{4}\frac{\partial^4}{\partial x^4}+\delta\gamma \frac{\partial^2}{\partial x^2}\Big)u(x,t) - 2\delta^2 u(x,0)\delta_0(t) ~~(\mbox{using}~\eqref{bmge})\\
&= \Big(\frac{1}{4}\frac{\partial^4}{\partial x^4}+\delta\gamma \frac{\partial^2}{\partial x^2}\Big)u(x,t)+w_t(x,0)h(0,t)+2\delta^2u(x,0)\delta_0(t)~~(\mbox{using Proposition}~\ref{dlim})
\end{align*}
and hence the result follows.
\end{proof}
\begin{remark}
When $\delta=1$ and $\gamma=0$, the following fractional PDE holds:
\begin{align*}
\sqrt{2} \left(\frac{\partial}{\partial x}\right)^{1/2}u(x,t) &= \int_{0}^{\infty}w(x,r)\sqrt{2}\left(\frac{\partial}{\partial x}\right)^{1/2}h(r,t)dr\\
&= -\int_{0}^{\infty}w(x,r)\frac{\partial}{\partial r}h(r,t)dr - \sqrt{\frac{2}{\pi t}}\int_{0}^{\infty}w(x,r)h(r,0)dr\\
&=w(x,0)h(0,t)+\frac{1}{2}\int_{0}^{\infty}\frac{\partial^2}{\partial x^2} w(x,r)h(r,t)dr - \sqrt{\frac{2}{\pi t}}u(x,0)\\
&= w(x,0)h(0,t)+\frac{1}{2}\frac{\partial^2}{\partial x^2}u(x,t)-\sqrt{\frac{2}{\pi t}}u(x,0).
\end{align*}
An interesting discussion for subordinated process with Brownian motion and a more general Markov process can be found in Baeumer {\it et al.}~(2009).
\end{remark}

\vone
\noindent{\bf Acknowledgments.}  The authors are grateful to the reviewers for several helpful suggestions and comments
which improved the quality of the paper.  The second author
(A. Kumar) wishes to thank Council of Scientific and Industrial Research (CSIR), India, for the award of a research fellowship.
\vtwo
\noindent {\bf \Large References}
\noindent
\begin{namelist}{xxx}

\item{} Abramowitz, M. and Stegun, I. A. (eds) (1992). {\it Handbook of Mathematical Functions with Formulas, Graphs and
Mathematical Tables}. Dover, New York.
\item{} Allouba, H. (2002). Brownian-time processes: The pde connection and the corresponding Feynman-Kac formula. {\it Trans. Amer. Math. Soc.} {\bf 354}, 4627-4637.
\item{} Allouba, H. and Zheng, W. (2001). Brownian-time processes: The pde connection and the half-derivative generator. {\it Ann. Probab.} {\bf 29}, 1780-1795.
\item{}  Applebaum, D. (2009). {\it L\'evy Processes and Stochastic Calculus}. 2nd ed., { Cambridge University Press},
Cambridge, U.K.

\item {} Baeumer, B. and Meerschaert, M. M. (2010). Tempered stable L\'évy motion and transient super-diffusion. {\it J. Comput. Appl. Math.} {\bf 233}, 2438-2448.

\item{} Baeumer, B. Meerschaert, M. M. and Nane, E. (2009). Brownian subordinators and fractional Cauchy problems. {\it Trans. Amer. Math. Soc.} {\bf 361}, 3915-3930.

\item{}  Barndorff-Nielsen, O. E. (1997). Normal inverse Gaussian distributions and stochastic volatility modeling.
{\it Scand. J. Statist.}, {\bf 24}, 1-13.

\item{}  Bertoin, J. (1996). {\it L\'evy Processes}. {Cambridge University Press}, Cambridge.

\item Boukai, B. (1990). An explicit expression for the distribution of the supremum of Brownian motion with a change point. {\it Comm. Statist. Theory Methods}. {\bf 19}, 31-40.

\item{}  Cont, R. and Tankov, P. (2004). {\it Financial Modeling with Jump Processes}. { Chapman \& Hall CRC Press}, Boca Raton.
\item{} Decreusefond, L. and Nualart, D (2008). Hitting times for Gaussian processes.
{\it Ann. Probab.} {\bf 36}, 319-330.
\item{}  Devroye, L. (1986). {\it Nonuniform Random Variate Generation}. {Springer}, New York.
\item{} Dufresne, F. and Gerber, H. U. (1993) The probability of ruin for the inverse
Gaussian and related processes. {\it Insurance Math. Econom.} {\bf 12}. 9-22.
\item {} Gerber, H. U. (1992). On the probability of ruin for infinitely divisible claim amount distributions. {\it Insurance Math. Econom.} {\bf 11}. 163-166.
\item{}  Halgreen, C. (1979). Self-decomposability of the generalized inverse Gaussian and hyperbolic distributions. {\it Z. Wahrsch. Verw.
 Gebiete.} {\bf 47}, 13-17.
\item{} Kumar, A., Meerschaert, M. M. and Vellaisamy, P. (2011). Fractional normal inverse Gaussian diffusion. {\it Statist. Probab. Lett.} {\bf 81}, 146-152.

\item {} Kumar, A. and Vellaisamy, P. (2012). Fractional normal inverse Gaussian process. {\it Methodol. Comput. Appl. Probab.}, {\bf 14}, 263-283.
\item{} Lee, M.-L. T. and Whitmore, G. A. (2006). Threshold regression for survival analysis: modeling event times by a stochastic process reaching a boundary. {\it Statist. Sci.}, {\bf 21}, 501-513.

\item{} Meerschaert, M. M. and Scheffler, H. (2004). Limit theorems for continuous-time random walks with infinite mean waiting times.
{\it J. App. Probab.} {\bf 41}, 623-638.

\item{} Meerschaert, M. M. and Scheffler, H. (2008). Triangular array limits for continuous time random walks.
{\it Stochastic Process. Appl.} {\bf 118}, 1606-1633.
\item{} Nane, E. (2008). Higher order PDE's and iterated processes. {\it Trans. Amer. Math. Soc.}, {\bf 360}, 2681-2692.
\item{} Olver, F. W. J. (1974). {\it Asymptotics and Special Functions}. Academic Press. New York.
\item{} Podlubny, I. (1999). {\it Fractional Differential Equations}. Academic Press. London. U. K.
\item{} Roberts, G. E. and Kaufman, H. (1966). {\it Table of Laplace Transforms.} W. B. Saunders. Philadelphia.
\item{} Samorodnitsky, G. and Taqqu, M. S. (2000). {\it Stable Non-Gaussian Random Processes: Stochastic Models with Infinite Variance.} CRC Press,
Boca Raton, Florida.
\item{} Sato, K. (1999). {\it L\'evy Processes and Infinitely Divisible Distributions}. Cambridge University Press.

\item{} Schiff, J. L. (1999). {\it The Laplace Transform: Theory and Applications}. Springer-Verlag, New York.
\item{}  Steutel, F.W. and Van Harn, K. (2004). {\it Infinite Divisibility of Probability Distributions on the Real Line.}
Marcel Dekker, New York.

\item{} Uchaikin, V. V. and Zolotarev, V. M. (1999). {\it Chance and Stability: Stable Distributions and Their Applications}.
VSP. Utrecht.
\item{} Veillette, M. and Taqqu, M. S. (2010a). Using differential equations to obtain joint moments of first-passage times of increasing L\'evy processes. {\it Statist. Probab. Lett.} {\bf 80}, 697-705.
\item{} Veillette, M. and Taqqu, M. S. (2010b). Numerical computation of first-passage times
of increasing L\'evy Processes. {\it Methodol. Comput. Appl. Probab.}, {\bf 12}, 695--729.



\end{namelist}

\end{document}